\documentclass{birkjour}
\usepackage{amscd,amsmath}
\usepackage{amssymb}
\usepackage{amsthm}
\usepackage{color}
\usepackage{hyperref}
\usepackage{enumerate}
\usepackage{geometry}
\theoremstyle{definition}
\newtheorem{theorem}{Theorem}[section]
\newtheorem{definition}[theorem]{Definition}
\newtheorem{proposition}[theorem]{Proposition}
\newtheorem{lemma}[theorem]{Lemma}
\newtheorem{corollary}[theorem]{Corollary}
\newtheorem{remark}[theorem]{Remark}
\newtheorem{example}[theorem]{Example}

\numberwithin{equation}{section}
\newtheorem{problem}{Problem}

\newtheorem{acknowledgements}{Acknowledgement}

\def\<{\left < }
\def\>{\right >}
\def\({\left ( }
\def\){\right )}

\def\C2{${\bf C}^2$}

\begin{document}

\title[Exist. and Nonexist. Warped Product Subman. of Almost Contact Man. \; \;  ]{Existence and Nonexistence of Warped Product Submanifolds of Almost Contact Manifolds} 

\author[A. Mustafa]{Abdulqader Mustafa}
\address{Department of Mathematics, Faculty of Arts and Science, Palestine Technical University, Kadoorei, Tulkarm, Palestine}
\email{abdulqader.mustafa@ptuk.edu.ps}
\author[C. $\ddot O$zel]{Cenap $\ddot O$zel}
\address{Department of Mathematics, Faculty of Science, King Abdulaziz University, 21589 Jeddah, Saudi Arabia}
\email{cozel@kau.edu.sa}
\;
\author[A. Pigazzini]{Alexander{\;}Pigazzini}
\address{Mathematical and Physical Science Foundation, 4200 Slagelse, Denmark}
\email{pigazzini@topositus.com}
\author[R. Pincak]{Richard Pincak}
\address{Institute of Experimental Physics, Slovak Academy of Sciences, Kosice, Slovak Republic}
\email{pincak@saske.sk}

\maketitle
\begin{abstract}
This paper has two goals; the first is to generalize results for the existence and nonexistence of warped product submanifolds of almost contact manifolds, accordingly a self-contained reference of such submanifolds is offered to save efforts of potential research. Most of the results of this paper are general and decisive enough to generalize both discovered and not discovered results. Moreover, a discrete example of contact $CR$-warped product submanifold in Kenmotsu manifold is constructed. For further research direction, we addressed a couple of open problems arose from the results of this paper.



\end{abstract}


\sloppy
\section{Introduction}
Warped products have been playing some important roles in the theory of general relativity as they have been providing the best mathematical models of our universe for now; that is, the warped product scheme was successfully applied in general relativity and semi-Riemannian geometry in order to build basic cosmological models for the universe. For instance, the Robertson-Walker spacetime, the Friedmann cosmological models and the standard static spacetime are given as warped product manifolds. For more cosmological applications, warped product manifolds provide excellent setting to model spacetime near black holes or bodies with large gravitational force. For example, the relativistic model of the Schwarzschild spacetime that describes the outer space around a massive star or a black hole admits a warped product construction \cite{iijj77}.  

In an attempt to construct manifolds of negative curvatures, R.L. Bishop and O'Neill \cite{ddyy7} introduced the notion of {\it warped product manifolds} as follows: Let $N_1$ and $N_2$ be two Riemannian manifolds with Riemannian metrics $g_{N_1}$ and $g_{N_2}$, respectively, and $f>0$ a $C^\infty$ function on $N_1$. Consider the product manifold $N_1\times N_2$ with its projections $\pi_1:N_1\times N_2\mapsto N_1$ and $\pi_2:N_1\times N_2\mapsto N_2$. Then, the {\it warped product} $\tilde M^m= N_1\times _fN_2$ is the Riemannian manifold $N_1\times N_2=(N_1\times N_2, \tilde g)$ equipped with a Riemannian structure such that 
 $\tilde g=g_{N_1} + f^2 g_{N_2}$.

A warped product manifold $\tilde M^m=N_1\times _fN_2$ is said to be {\it trivial} if the warping function $f$ is constant. For a nontrivial warped product $N_1\times _fN_2$, we denote by $\mathfrak{D}_1$ and $\mathfrak{D}_2$ the distributions given by the vectors tangent to leaves and fibers, respectively. Thus, $\mathfrak{D}_1$ is obtained from tangent vectors of $N_1$ via the horizontal lift and $\mathfrak{D}_2$ is obtained by tangent vectors of $N_2$ via the vertical lift. 

Since our goal to search about existence and nonexistence of warped product submanifolds in almost contact manifolds, we hypothesize the following two problems. The first is for single warped products
\begin{problem}\label{prob9}
Prove existence or nonexistence of single warped product submanifolds of almost contact manifolds.
\end{problem}

The second problem is for doubly warped products

\begin{problem}\label{prob10}
Prove existence or nonexistence of doubly warped product submanifolds of almost contact manifolds.
\end{problem}

The present paper is organized as follows: After the introduction, we present in Section 2, the preliminaries, basic definitions and formulas. In Section 3, we provide basic results, which are necessary and useful to the next section. In Section 4, which is the main section, we generalize theorems for existence and nonexistence warped product submanifolds for single and doubly warped product submanifolds in almost contact  manifolds. Moreover, in the current section we discuss the contact $CR$-warped product submaifolds in almost contact manifolds and construct an example of both types of contact $CR$-warped product submanifolds of Kenmotsu manifolds. In the final section, we address two open problems related to the obtained results in this paper. 

\section{Preliminaries}

At first, let us recall the following important two facts regarding Riemannian submanifolds, \cite{was}.
\begin{definition}\label{dfffff5555}
Let $M^n$ and $\tilde M^m$ be differentiable manifolds. A differentiable mapping $\varphi: M^n\longrightarrow \tilde M^m$ is said to be an {\it immersion} if $d\varphi_x: T_xM^n\rightarrow T_{\varphi (x)} \tilde M^m$ is injective for all $x\in M^n$. If, in addition, $\varphi$ is a homeomorphism onto $\varphi(M^n)\subset \tilde M^m$, where $\varphi(M^n)$ has the subspace topology induced from $\tilde M^m$, we say that $\varphi$ is an {\it embedding}. If $M^n\subset \tilde M^m$ and the inclusion $\boldsymbol{i}: M^n\subset \tilde M^m$ is an embedding, we say that $M^n$ is a submanifold of $\tilde M^m$.
\end{definition}
It can be seen that if $\varphi: M^n\rightarrow \tilde M^m$ is an immersion, then $n\le m$; the difference $m-n$ is called the {\it codimension} of the immersion $\varphi$.

For most local questions of geometry, it is the same to work with either immersions or embeddings. This comes from the following proposition which shows that every immersion is locally (in a certain sense) an embedding.

\begin{proposition}\label {3g9stg}
Let $\varphi: M^n\longrightarrow \tilde M^m$, $n\le m$, be an immersion of the differentiable manifold $M^n$ into the differentiable manifold $\tilde M^m$. For every point $x\in  M^n$, there exists a neighborhood ${\mathfrak u}$ of $x$ such that the restriction $\varphi | {\mathfrak u}\rightarrow \tilde M^m$ is an embedding.
\end{proposition}

Now, we turn our attention to the differential geometry of the submanifold theory. First, let $M^n$ be $n$-dimensional Riemannian manifold isometrically immersed in an $m$-dimensional Riemannian manifold $\tilde M^m$. Since we are dealing with a local study, then, by Proposition \ref{3g9stg}, we may assume that $M^n$ is embedded in $\tilde M^m$. On this infinitesimal scale, Definition \ref{dfffff5555} guarantees that $M^n$ is a {\it Riemannian submanifold} of some nearby points in $\tilde M^m$ with induced Riemannian metric $g$. Then, {\it Gauss} and {\it Weingarten formulas} are, respectively, given by
\begin{equation}\label{3}
\tilde \nabla_X Y=\nabla_X Y+h(X,Y) 
\end{equation}
and
\begin{equation}\label{4}
\tilde\nabla_X\zeta=-A_\zeta X+\nabla^\perp_X\zeta 
\end{equation}
for all $X,Y\in \Gamma(TM^n)$ and $\zeta\in \Gamma (T^\perp M^n)$, where $\tilde\nabla$ and $\nabla$ denote respectively the Levi-Civita and the {\it induced} Levi-Civita connections on $\tilde M^m$ and $M^n$, and $\Gamma(TM^n)$ is the module of differentiable sections of the vector bundle $TM^n$. $\nabla^\perp$ is the {\it normal connection} acting on the normal bundle $T^\perp M^n$. 

Here, $g$ denotes the {\it induced Riemannian metric} from $\tilde g$ on $M^n$. For simplicity's sake, the inner products which are carried by $g$, $\tilde g$ or any other induced Riemannian metric are performed via $g$. However, most of the inner products which will be applied in this thesis are equipped with $g$, other situations are rarely considered.

Here, it is well-known that the {\it second fundamental form} $h$ and the {\it shape operator} $A_\zeta$ of $M^n$ are related by  
\begin{equation}\label{5}
g(A_\zeta X,Y)=g(h(X,Y),\zeta)
\end{equation}
for all $X,Y\in \Gamma(TM^n)$ and $\zeta\in \Gamma(T^\perp M^n)$, \cite{pom},~\cite{iijj77}.

Geometrically, $M^n$ is called a {\it totally geodesic} submanifold in $\tilde M^m$ if $h$ vanishes identically. Particularly, the {\it relative null space}, ${\mathcal N}_x$, of the submanifold $M^n$ in the Riemannian manifold $\tilde M^m$ is defined at a point $x\in M^n$ by \cite{aallr4} as
\begin{equation}\label{19}
{\mathcal N}_x=\{ X\in T_xM^n: h(X, Y)=0~~~ \forall~ Y\in T_xM^n\}.
\end{equation}

In a different line of thought, and for any $X\in \Gamma (TM^n)$, $\zeta\in \Gamma (T^\perp M^n)$ and a $(1,1)$ tensor field $\psi$ on $\tilde M^m$, we write
\begin{equation}\label{6}
\psi X=PX+FX,
\end{equation}
and
\begin{equation}\label{7}
\psi N=t\zeta+f\zeta,
\end{equation}
where $PX$, $t\zeta$ are the tangential components and $FX$, $f\zeta$ are the normal components of $\psi X$ and $\psi \zeta$, respectively, \cite{saw}. In the sake of following the common terminology, the tensor field $\psi$ is replaced by $J$ in almost Hermitian manifolds. However, the covariant derivatives of the tensor fields $\psi$, $P$ and $F$ are respectively defined as \cite{pom}
\begin{equation}\label{42}
(\tilde\nabla_X\psi)Y=\tilde\nabla_X\psi Y-\psi\tilde\nabla_XY,
\end{equation}
\begin{equation}\label{38}
(\tilde\nabla_XP)Y=\tilde\nabla_XPY-P\tilde\nabla_XY
\end{equation}
and
\begin{equation}\label{39}
(\tilde\nabla_XF)Y=\nabla_{X}^{\perp}FY-F\tilde\nabla_XY.
\end{equation}

Likewise, we consider a local field of orthonormal frames 
$\{e_1, \cdots , e_n, e_{n+1}, \cdots, e_m\}$ on $\tilde M^m$, such that, restricted to $M^n$, $\{e_1, \cdots , e_n\}$ are tangent to $M^n$ and $\{e_{n+1}, \cdots, e_m\}$ are normal to $M^n$. Then, the {\it mean curvature vector} $\vec H(x)$ is introduced as \cite{pom}, \cite{iijj77}
\begin{equation}\label{8}
\vec H(x)=\frac{1}{n} \sum_{i=1}^{n} h(e_i, e_i),
\end{equation}

On one hand, we say that $M^n$ is a {\it minimal submanifold} of $\tilde M^m$ if $\vec H=0$. On the other hand, one may deduce that $M^n$ is totally umbilical in $\tilde M^m$ if and only if $h(X,Y)=g(X,Y) \vec H$, for any $X,~Y\in \Gamma (TM^n)$ \cite{yyhh88}, where $H$ and $h$ are the mean curvature vector and the second fundamental form, respectively \cite{2233ee}.

For an odd dimensional real $C^\infty$ manifold $\tilde M^{2l+1}$, let $\phi$, $\xi$, $\eta$ and $\tilde g$ be respectively a $(1,~1)$ tensor field, a vector field, a $1$-form and a Riemannian metric on $\tilde M^{2l+1}$ satisfying
\begin{equation}\label{151}
\left.
\begin{aligned}
   \phi^2=-I+\eta\otimes\xi,~~~~~\phi\xi=0,~~~~~\eta\circ\phi=0,~~~~~\eta(\xi) = 1,\\ 
   \eta(X)=\tilde g(X, \xi),~~~~~\tilde g(\phi X, \phi Y)=\tilde g(X, Y)-\eta(X)\eta(Y),
\end{aligned}
\right\}
\end{equation}
for any $X,~Y\in \Gamma (T\tilde M^{2l+1})$. Then we call $(\tilde M^{2l+1}, \phi, \xi, \eta, \tilde g)$ an {\it almost contact metric manifold} and $(\phi, \xi, \eta, \tilde g)$ an {\it almost contact metric structure} on $\tilde M^{2l+1}$, see \cite{pom},~\cite{rash} and \cite{fottt}.

A fundamental $2$-form $\Phi$ is defined on $\tilde M^{2l+1}$ by $\Phi(X,Y)=\tilde g(\phi X,Y)$. An almost contact metric manifold $\tilde M^{2l+1}$ is called a contact metric manifold if $\Phi =\frac{1}{2} d\eta$. If the almost contact metric manifold $(\tilde M^{2l+1},\phi, \xi, \eta, \tilde g)$ satisfies $[\phi, \phi]+2d\eta\otimes \xi=0$, then $(\tilde M^{2l+1},\phi, \xi, \eta, \tilde g)$ turns out to be a {\it normal almost contact manifold}, where the Nijenhuis tensor is defined as
$$[\phi,\phi](X,Y)=[\phi X, \phi Y]+\phi^2[X,Y]-\phi [X,\phi Y]-\phi[\phi X, Y]~~~~~\forall~ X,~Y\in \Gamma(T\tilde M^{2l+1}).$$

For our purpose, we will distinguish four classes of almost contact metric structures; namely, Sasakian, Kenmotsu, cosymplectic and nearly trans-Sasakian structures. At first, an almost contact metric structure is  is said to be {\it Sasakian} whenever it is both contact metric and normal, equivalently \cite{wenn}
\begin{equation}\label{157}
(\tilde\nabla_X\phi)Y=-\tilde g(X,Y)\xi+\eta(Y)X.
\end{equation}

 An almost contact metric manifold $\tilde M^{2l+1}$ is called {\it Kenmotsu manifold} \cite{foss} if 
\begin{equation}\label{155}
(\tilde\nabla_X\phi)Y=\tilde g(\phi X, Y)\xi -\eta(Y)\phi X.
\end{equation}

In the case of killing almost contact structure tensors, consider a normal almost contact metric structure $(\phi, \xi,\eta,\tilde g)$ with both $\Phi$ and $\eta$ are closed. Then, such $(\phi, \xi,\eta,\tilde g)$ is called {\it cosymplectic} \cite{nasso}. Explicitly, cosymplectic manifolds are characterized by normality and the vanishing of Riemannian covariant derivative of $\phi$, i.e.,
\begin{equation}\label{152}
(\tilde\nabla_X\phi)Y=0.
\end{equation}
 Hereafter, we call the almost contact manifold $\tilde M^{2l+1}$ a {\it nearly cosymplectic } manifold if 
\begin{equation}\label{154}
(\tilde \nabla_X\phi)Y+(\tilde \nabla_Y\phi)X=0.
\end{equation}

Based on Gray-Hervella classification of almost Hermitian manifolds \cite{ogma}, an almost contact metric structure $(\phi , \xi , \eta , \tilde g)$ on $\tilde M^{2l+1}$ is called a trans-Sasakian structure \cite{zeedo} if $(\tilde M^{2l+1}\times \mathbb{R}, J, \tilde G)$ belongs to the class $W_4$ of their classification, where $J$ is the almost complex structure on $\tilde M^{2l+1}\times \mathbb{R}$ defined by
$$J(X, ad/dt)= \biggl(\phi X- a\xi, \eta (X)d/dt\biggr)$$
for all vector fields $X$ on $\tilde M^{2l+1}$ and smooth functions $a$ on $\tilde M^{2l+1}\times \mathbb{R}$, where $\tilde G$ is the product metric on $\tilde M^{2l+1}\times \mathbb{R}$. This may be expressed by the condition 
\begin{equation}\label{274}
(\tilde \nabla_X\phi)Y= \alpha \biggl(\tilde g(X,Y)\xi - \eta (Y) X\biggr) + \beta \biggl(\tilde g(\phi X, Y)\xi - \eta (Y) \phi X\biggr),
\end{equation}
for some smooth functions $\alpha$ and $\beta$ on $\tilde M^{2l+1}$, and we say that the trans-Sasakian structure is of type $(\alpha, \beta)$. From the above formula it follows that 
$$\tilde \nabla_X\xi= - \alpha \phi X + \beta \biggl(X- \eta (X) \xi \biggr).$$
Up to D. Chinea and C. Gonzalez classification of almost contact structures \cite{ches}, the class $C_6\otimes C_5$ coincides with the class of trans-Sasakian structure of type $(\alpha , \beta)$. Recently, J. C. Marrero proved that a trans-Sasakian manifold of dimension $\geq$5 is either $\alpha$-Sasakian, $\beta$-Kemnotsu or a cosymplectic manifold,  \cite{massa}.

In \cite{zeedo}, C. Gherghe introduced  nearly trans-Sasakian structure of type $(\alpha , \beta )$. An almost contact metric structure $(\phi , \xi , \eta ,\tilde g)$ on $\tilde M^{2l+1}$ is called a {\it nearly trans-Sasakian} structure (Mustafa et al., 2014 $\&$ 2015)   if 
\begin{equation}\label{261}
(\tilde \nabla _X \phi ) Y + (\tilde \nabla _Y \phi ) X= \alpha \biggl(2\tilde g(X, Y) \xi - \eta (Y) X - \eta (X) Y\biggr)$$$$-\beta \biggl( \eta (Y) \phi X + \eta (X) \phi Y\biggr).
\end{equation}
 Evidently, a nearly trans-Sasakian of type $(\alpha , \beta)$ is nearly-Sasakian, nearly Kenmotsu or nearly cosymplectic according as $\beta$ = 0, $\alpha$=1; or $\alpha$ = 0, $\beta$=1;  or $\alpha$ = $\beta$ = 0, respectively.

\section{Basic Lemmas}

To relate the calculus of $N_1\times N_2$ to that of its factors the crucial notion of {\it lifting} is introduced as follows. If $f\in {\mathfrak F}(N_1)$, the {\it lift} of $f$ to $N_1\times N_2$ is $\tilde f=f\circ \pi_1\in {\mathfrak F}(N_1\times N_2)$. If $X_p\in T_p(N_1)$ and $q\in N_2$, then the {\it lift} $ X_{(p,q)}$ of $X_p$ to $(p,q)$ is the unique vector in $T_{(p,q)}(N_1)$ such that $d\pi_1( X_{(p,q)})=X_p$. If $X\in \Gamma(TN_1)$ the {\it lift} of $X$ to $N_1\times N_2$ is the vector field $X$ whose value at each $(p,q)$ is the lift of $X_p$ to $(p,q)$. The set of all such {\it horizontal lifts} $ X$ is denoted by ${\mathcal L} (N_1)$. Functions, tangent vectors and vector fields on $N_2$ are lifted to $N_1\times N_2$ in the same way using the projection $\pi_2$. Note that ${\mathcal L} (N_1)$ and symmetrically the {\it vertical lifts} ${\mathcal L} (N_2)$ are vector subspaces of $\Gamma \bigl(T(N_1\times N_2)\bigr)$, \cite{iijj77}. 

We recall the following two general results for warped products \cite{iijj77}.
 \begin{proposition}\label{1}On $\tilde M^m = N_1\times_{f}N_2$, if $X,~Y\in {\mathcal L}(N_1)$ and $Z,~W\in {\mathcal L} (N_2)$, then
\begin{itemize}
\item[(i)] $\tilde\nabla_XY\in {\mathcal L}(N_1)$ is the lift of $\tilde \nabla_XY$ on $N_1$.
\item[(ii)]$\tilde \nabla_XZ=\tilde \nabla_ZX=(Xf/f)Z.$
\item[(iii)] $(\tilde\nabla_ZW)^\perp = h_{N_2}(Z,W)= - \bigl(g_{N_2}(Z, W)/f\bigr) \nabla (f).$
\item[(iv)] $(\tilde \nabla_ZW)^T\in {\mathcal L}(N_2)$ is the lift of $\nabla^{N_2}_ZW$ on $N_2$,
\end{itemize}
where $g_{N_2}$, $h_{N_2}$ and $\nabla^{N_2}$ are, respectively, the induced Riemannian metric on $N_2$, the second fundamental form of $N_2$ as a submanifold of $\tilde M^m$ and the induced Levi-Civita connection on $N_2$. \footnotemark[\value{footnote}]\footnotetext{The operators $\perp$, $T$ and $\nabla (f)$ refer to the normal projection, the tangential projection and the gradient of $f$, respectively.} 
\end{proposition}

It is obvious that, the above proposition leads to the following geometric conclusion.
\begin{corollary}\label{2}
 The leaves $N_1\times q$ of a warped product are totally geodesic; the fibers $p\times N_2$ are totally umbilical.
\end{corollary}

Clearly, the totally geodesy of the leaves follows from $(i)$, while $(iii)$ implies that the fibers are totally umbilical in $\tilde M^m$. It is significant to say that, this corollary is one of the key ingredients of this work. Since all our considered submanifolds are warped products.

Here, it is well-known that the {\it second fundamental form} $\sigma$ and the {\it shape operator} $A_\xi$ of $M^n$ are related by  
\begin{equation}\label{5}
g(A_\xi X,Y)=g(\sigma(X,Y),\xi)
\end{equation}
for all $X,Y\in \Gamma(TM^n)$ and $\xi\in \Gamma(T^\perp M^n)$ (for instance, see \cite{pom}, \cite{iijj77}).

\section{Existence and Nonexistence of Warped Product Submanifolds in Almost Contact Manifolds }

This section has two significant purposes. The first one is to provide special case solutions for Problems \ref{prob9} and \ref{prob10}, that is to see whether a warped product exists or not in almost contact  manifolds. In the existence case, we prove some preparatory characteristic results which are necessary for subsequent sections, and this is the second purpose. Some new example is given to assert the existence of some important warped product manifolds. 

For a submanifold $M^n$ in an almost contact manifold $\tilde M^{2m+1}$  let ${\mathcal P}_XY$ denote the tangential component and ${\mathcal Q}_XY$ the normal one of $(\tilde\nabla_X\phi)Y$  in $\tilde M^{2m+1}$, where $X,~Y\in \Gamma (TM^n)$. 

In order to make it a self-contained reference of warped product submanifolds for immersibility and nonimmersibility problems, we hypothesize most of our statements in the current and the next section for almost contact manifolds, and for warped product submanifolds of type $N_T\times _fN_2$, where $N_T$ and $N$ are holomorphic and Riemannian submanifolds. Meaning that, a lot of particular case results are included in the theorems of the next section.

It is still an open question whether or not a warped product admits isometric immersions into certain Riemannian manifolds of interest. For instance, many articles have been recently published in almost contact manifolds (see, for example \cite{xos} and \cite{wenal}). In fact, these papers and a lot others (see references in \cite{ssee44}) provide special case answers for Problems \ref{prob9} and \ref{prob10}. The following theorem generalizes all such nonexistence results as a final answer for doubly warped product submanifolds in almost contact manifolds. 
\begin{theorem}\label{ndo}
In almost contact manifolds, there does not exist a proper doubly warped product submanifold $M^n=_{f_2}{N_1}\times _{f_1}{N_2}$ such that the characteristic vector field $\xi$ is either tangent to $N_1$ or $N_2$.
\end{theorem} 
\begin{proof}
Suppose $\xi$ in $\Gamma(TN_2)$. Then for any $X\in \Gamma(TN_1)$,  we directly calculate
$$2X\ln f_1=2X\ln f_1 g(\xi, \xi)=2g(\tilde\nabla_X\xi, \xi)=Xg(\xi, \xi)=X(1)=0.$$
This means that $f_1$ is constant. Similarly, it can be shown that $f_2$ is constant when $\xi$ is tangent to the first factor. Hence, we conclude that a doubly warped product submanifold of almost contact manifolds, in the sense of our hypothesis, is trivial, which completes the proof.
\end{proof}

Considering $\xi$ as in the above hypothesis, this theorem can be simply paraphrased by saying that: doubly warped product submanifolds in almost contact manifolds are but trivial. With this fact, some results concerning inequalities for doubly warped product submanifolds in Kenmotsu manifolds become trivial (see references in \cite{ssee44}). 

As a special case of Theorem \ref{ndo}, we have the following theorem for (singly) warped product submanifolds
\begin{theorem}\label{fund}
There is no warped product submanifolds in almost contact manifolds such that the characteristic vector field $\xi$ is tangent to the second factor. 
\end{theorem}

The above theorem answers some special cases of Problems \ref{prob9} and \ref{prob10}. On one hand, it generalizes all related nonexistence results of this topic (see, for example \cite{xos},\cite{wenal}, \cite{zolo}, \cite{fonm} and \cite{hasss}). On the other hand, it guides us to restrict the choice of the factor that $\xi$ should be tangent to in warped product submanifolds of almost contact manifolds.

From now on, the characteristic vector field $\xi$ is supposed to be tangent to the first factor of all warped product submanifolds in almost contact manifolds. Henceforth, it s strightforward to get
\begin{theorem}\label{jop}
For each warped product submanifold $N\times _fN_T$ of almost contact manifolds such that $\xi$ is tangent to the first factor, the following are true
\begin{enumerate}
\item[(i)] $g({\mathcal P}_XZ,W)=0;$
\item[(ii)] $g({\mathcal P}_ZX,JZ)-g({\mathcal P}_{JZ}X,Z)=-2 (X\ln f) ||Z||^2,$
\end{enumerate}
for every vector field $X\in \Gamma(TN)$, and $Z,~W\in \Gamma (TN_T)$.
\end{theorem}

As a direct application of the preceding theorem, and by using $(\ref{261})$, we state the following remark, which generalizes a lot of nonexistence results in almost contact manifolds (see, for example \cite{wsssa} and \cite{wenal}). First, by putting $\beta=0$ in $(\ref{261})$, we get the structural formula for nearly $\alpha$-Sasakian manifolds; that is,
\begin{equation}\label{2ms1}
(\tilde \nabla _X \phi ) Y + (\tilde \nabla _Y \phi ) X= \alpha \biggl(2\tilde g(X, Y) \xi - \eta (Y) X - \eta (X) Y\biggr).
\end{equation}

Now, we show that the first term on the left hand side of statement $(ii)$ above is zero. From the above equation, we directly get 
 $$g({\mathcal P}_ZX,JZ)=-g({\mathcal P}_XZ,JZ)-\alpha\eta(X)g(Z,JZ).$$
In view of statement $(i)$ of the above theorem, the right hand side of the above equation vanishes identically. Similarly, we can show that $g({\mathcal P}_{JZ}X,Z)=0$. Hence, statement $(ii)$ implies that $X\ln f=0$. This also holds for nearly cosyplectic manifolds, one can prove that using similar analogy like above. Thus, we have the following
\begin{remark}\label{fat}
Warped products of the type $N\times _fN_T$ do not exist in nearly Sasakian and nearly cosymplectic manifolds if $\xi$ is tangent to the first factor, and so for Sasakian and cosymplectic manifolds. However, the situation is different in Kenmotsu manifolds as we will see in the following example and in the next chapter also.  
\end{remark}

A submanifold $M^n$ of an almost contact metric manifold $\tilde M^{2l+1}$ is said to be a {\it contact CR-submanifold }if there exist on $M^n$ differentiable distributions $\mathfrak{D}_T$ and $\mathfrak{D}_\perp$, satisfying the following
\begin{enumerate}
\item [(i)] $TM^n=\mathfrak{D}_T\oplus \mathfrak{D}_\perp\oplus \langle\xi\rangle$,
\item [(ii)] $\mathfrak{D}_T$ is an invariant distribution, i.e., $ \phi (\mathfrak{D}_T) \subseteq \mathfrak{D}_T$,
\item [(iii)] $\mathfrak{D}_\perp$ is an anti-invariant distribution, i.e., $\phi (\mathfrak{D}_\perp) \subseteq T^\perp M^n$.
\end{enumerate}

In Sasakian manifolds, a concrete example of contact $CR$-warped product submanifolds of the type $N_T\times _fN_\perp$ can be found in \cite{wenal}. On the contrary, and in view of Remark \ref{fat}, we conclude that warped product submanifolds with second invariant factor are trivial in both Sasakian and cosymplectic manifolds when $\xi$ is tangent to the first factor. In particular, this implies that contact $CR$-warped product submanifolds of the type $N_\perp\times _fN_T$ reduces to be contact $CR$-products in Sasakian and cosymplectic manifolds. By contrast, such warped product submanifolds do exist in Kenmotsu manifolds. 

To assert the above claim, we provide a counter example that ensures such existence of warped product submanifolds in Kenmotsu manifolds when the second factor is invariant. Besides, we can get an insurance for the existence of contact $CR$-warped product submanifolds in Kenmotsu manifolds, for both types; $M^n=N_T\times _fN_\perp$ and $M^n=N_\perp\times _fN_T$, when $\xi$ is tangent to the first factor. 
\begin{example}\label{S1}
Let $\tilde M^9=\mathbb{R}\times _{e^t}\mathbb{C}^4$ be a Kenmotsu manifold, where $\mathbb{R}$ is the real line, and $\mathbb{C}^4$ is a Kaehler manifold with Kaehlerian structure $(G,J)$. Here, $G$ and $J$ are the restrictions of $g$ and $\phi$ to $\tilde M^9(p)$, respectively, for every $p\in \tilde M^9$. Let $(t, x_1,\cdots, x_8)$ be a local coordinates frame of $\tilde M^9$ where $t$ and $(x_1, \cdots, x_8)$ denote the local coordinates of $\mathbb{R}$ and $\mathbb{C}^4$, respectively. It is well-known that the Riemannian metric tensor $g$ and the vector field $\xi$ are defined on $\tilde M^9$ as follows \cite{foss}:
\[ g_{(t,x)}=\begin{pmatrix}
1 & 0 \\
0 & e^{2t}G_(x) \\
 \end{pmatrix},~~~~~~~\xi=\left(\frac{d}{dt}\right). \]

Now, consider the three-dimensional submanifold $M^3$ of $\mathbb{C}^4$ given by the equations
$$x_1=e^t v,~~~x_2=e^tu,~~~x_3=e^tv,~~~x_4=e^tu,~~~x_5=e^ts,~~~x_7=e^ts,~~~x_6=x_8=0.$$
Observe that the tangent bundle $TM^3$ is spanned by $Z_1$, $Z_2$ and $Z_3$, where
$$Z_1=e^t\frac{\partial}{\partial x_1}+e^t\frac{\partial}{\partial x_3},~~~
Z_2=e^t\frac{\partial}{\partial x_2}+e^t\frac{\partial}{\partial x_4},~~~
Z_3=e^t\frac{\partial}{\partial x_5}+e^t\frac{\partial}{\partial x_7}.$$

Further, we define the distributions $\mathfrak{D}_T=$span$\{Z_1,~Z_2\}$, and $\mathfrak{D}_\perp=$span$\{Z_3\}$. It is obvious that $\mathfrak{D}_T$ and $\mathfrak{D}_\perp$ are holomorphic and totally real distributions on $\mathbb{C}^4$, respectively. Hence, and taking into consideration $\phi (\xi)=0$, the distributions $\mathfrak{D}_\perp \oplus\langle\xi\rangle$ and $\mathfrak{D}_T$ are respectively anti-invariant and invariant distributions on $\tilde M^9$. Thus, $N^4=\mathfrak{D}_\perp \oplus\langle\xi\rangle\oplus \mathfrak{D}_T$ is a contact $CR$-submanifold in $\tilde M^9$. In addition, it is easy to see that both $\mathfrak{D}_\perp \oplus\langle\xi\rangle$ and $\mathfrak{D}_T$ are integrable. If we denote by $N_\perp$ and $N_T$ the integral manifolds of $\mathfrak{D}_\perp \oplus\langle\xi\rangle$ and $\mathfrak{D}_T$, respectively, then the metric tensor $g$ of $N^4$ is 
$$g=dt^2+e^{2t} ds^2+e^{2t}(dv^2+du^2)
=g_{N_\perp}+e^{2t}g_{N_T}.$$
Therefor, $N^4$ is a contact $CR$-warped product submanifold of $\tilde M^9$ of the type $N_\perp\times _fN_T$ with warping function $f=e^t$. Moreover, it straight forward to figure out that
$$h(Z_1, Z_1)=h(Z_2, Z_2)=0.$$
Hence, $N^4$ is a $\mathfrak{D}_2$-minimal warped product submanifold as expected, where $\mathfrak{D}_2=\mathfrak{D}_T$..

Likewise, by an analogous procedure to the above we can deduce that $\mathfrak{D}_T\oplus \langle\xi\rangle$ is an invariant distribution on $\tilde M^9$, and $\mathfrak{D}_\perp$ is an anti-invariant. Also, it is not difficult to show integrability of $\mathfrak{D}_T\oplus \langle\xi\rangle$. Denoting the integral manifolds of $\mathfrak{D}_T\oplus \langle\xi\rangle$ and $\mathfrak{D}_\perp$ by $N_T$ and $N_\perp$, respectively, we find that $N^4=N_T\times _{e^t}N_\perp$ is a non-trivial contact $CR$-warped product in $\tilde M^9$. By calculating the coefficients of $h$ restricted to $N_T$, we deduce that $N^4=N_T\times _{e^t}N_\perp$ is a $\mathfrak{D}_1$-minimal warped product submanifold as it should be, where $\mathfrak{D}_1=\mathfrak{D}_T$.
\end{example}

In this sequel, proper warped product submanifolds of types $N_\theta\times _fN_T$ and $N_T\times _fN_\theta$ do exist in Kenmotsu manifolds, when $\xi$ is tangent to the first factor. Whereas, Remark \ref{fat} informs us that proper warped product submanifolds of type $N_\theta\times _fN_T$ do not exist in both Sasakian and cosymplectic manifolds. Soon we show the nonexistence of $N_T\times _fN_\theta$ in Sasakian and cosymplectic manifolds such that $N_\theta$ is proper slant.

The following theorem is so significant because it will be used in the rest of this work. 
\begin{theorem}\label{T1}
Let $M^n=N_T\times _fN$ be a warped product submanifold isometrically immersed in an almost contact manifold $\tilde M^{2l+1}$ such that $\xi$ is tangent to the first factor. Then, we have the following 
\begin{enumerate}
\item[(i)] $g({\mathcal P}_XZ, Y)=- g(h(X,Y), FZ);$
\item[(ii)] $g({\mathcal P}_ZX, Z)=(\phi X\ln f)||Z||^2+g(h(X,Z), FZ);$
\item[(iii)] $g({\mathcal P}_ZX, Y)=0;$
\item[(iv)] $g({\mathcal P}_ZX, W)+g({\mathcal P}_WX, Z)=2(\phi X\ln f)g(Z,W)\newline~~~~~~~~~~~~~~~~~~~~+g(h(X,Z), FW)+g(h(X,W), FZ);$
\item[(v)] $ g({\mathcal P}_ZX-{\mathcal P}_XZ,W)-g({\mathcal P}_WX, Z)=2 (X\ln f) g(Z, PW);$
\item[(vi)] $g({\mathcal P}_XZ, W)+g({\mathcal P}_XW, Z)=0;$
\item[(vii)] $g({\mathcal Q}_XX, \phi \zeta)+g({\mathcal Q}_{\phi X}\phi X, \phi  \zeta)=-g(h(X,X), \zeta)-g(h(\phi X, \phi X), \zeta),$
\end{enumerate}
for arbitrary vector fields $X,~Y\in \Gamma(TN_T)$, $Z,~W\in \Gamma (TN)$ and $\zeta \in \Gamma (\nu)$.
\end{theorem}
\begin{proof}
The assertion of statements $(i),~(ii),~(iv),~(v)$ and $(vi)$ are trivial. For statement $(iii)$, suppose that $X$ and $Z$ are taken as hypothesis. Then it is obvious that
\begin{equation}\label{bil}(\tilde \nabla_X\phi)Z=\tilde \nabla_XPZ+\tilde \nabla_XFZ-\phi\tilde \nabla_XZ.\end{equation}
Also, for $X$ and $Z$ we have 
\begin{equation}\label{bel}(\tilde \nabla_Z\phi)X=\tilde \nabla_Z\phi X-\phi\tilde \nabla_ZX.\end{equation}
By subtracting $(\ref{bel})$ from $(\ref{bil})$, we obtain 
$$(\tilde \nabla_X\phi)Z-(\tilde \nabla_Z\phi)X=\tilde \nabla_XPZ+\tilde \nabla_XFZ-\tilde \nabla_Z\phi X.$$
Taking the inner product by $\phi Y$ in the above equation, gives
$$g({\mathcal P}_XZ, \phi Y)-g({\mathcal P}_ZX, \phi Y)=-g(h(X,\phi Y), FZ).$$
Replacing $\phi Y$ by $Y$ yields 
$$g({\mathcal P}_ZX, Y)-\eta(Y)g({\mathcal P}_ZX, \xi)-g({\mathcal P}_XZ, Y)+\eta(Y)g({\mathcal P}_XZ,\xi)=$$$$g(h(X,Y), FZ)-\eta(Y)g(h(X,\xi),FZ).$$
By using $(i)$ in the above equation we derive
 $$g({\mathcal P}_ZX, Y)=\eta(Y)g({\mathcal P}_ZX, \xi).$$
Since the right hand side of the above equation vanishes identically, we obtain $(iii)$. 

For $(vii)$, if we take $X=\xi$ in the above theorem, then statement $(vii)$ holds directly. Now, for an arbitrary vector field tangential to the first factor and perpendicular to $\xi$, say $X$, we have

 $$(\tilde \nabla_X\phi)X=\tilde \nabla_X\phi X-\phi\tilde \nabla_XX.$$
First, take the inner product in the above equation with $\phi\zeta$ to get 
 $$g({\mathcal Q}_XX, \phi\zeta)=g(h(\phi X,X), \phi\zeta)-g(h(X, X), \zeta).$$
\noindent
After that, we replace $\phi X$ by $X$ in the above equation to derive
$$g({\mathcal Q}_{\phi X}\phi X, \phi \zeta)=-g(h(\phi X,X), \phi\zeta)-g(h(\phi X, \phi X), \zeta).$$
Hence $(vii)$ can be obtained by adding the above two equations.
\end{proof}

In virtue of Theorem \ref{T1} $(v)$, we get the following decisive nonexistence result in the setting of almost contact structures, which generalizes several nonexistence results in this field (see references in \cite{ssee44}. 
\begin{corollary}\label{Y1}
In both of Sasakian and cosymplectic manifolds, there is no warped product submanifolds with invariant first factor tangential to $\xi$, other than contact $CR$-warped products.
\end{corollary}

In particular, this corollary implies the nonexistence of warped product submanifolds of type $N_T\times _fN_\theta$ in Sasakian and cosymplectic manifolds such that $N_\theta$ is a proper slant. On the contrary, this is not true for Kenmotsu manifolds.

Now, we prepare the following results for later use.
\begin{theorem}\label{U1}
Let $M^n=N_1\times _fN_2$ be a warped product submanifold isometrically immersed in a nearly trans-Sasakian manifold $\tilde M^{2l+1}$ such that $\xi$ is tangent to $N_1$. Then, the following hold
\begin{enumerate}
\item[(i)] $\xi \ln f=\beta;$
\item[(ii)] $g(h(\xi, Z), FZ)=-\alpha ||Z||^2,$
\end{enumerate}
for each vector field $Z$ tangent to $N_2$.
\end{theorem}
\begin{proof}
By $(\ref{261})$, it is straightforward that
\begin{equation}\label{psb}
-\phi \tilde \nabla_Z\xi+\tilde \nabla_\xi \phi Z-\phi \tilde \nabla_\xi Z=-\alpha Z-\beta \phi Z.
\end{equation}
For $(i)$, taking the inner product with $\phi Z$ in the above equation, gives
$$-2~\xi \ln f ||Z||^2+g(\tilde \nabla_\xi\phi Z,\phi Z)=-\beta ||Z||^2,$$
Equivalently,
$$-2~\xi \ln f ||Z||^2+\frac{1}{2}~\xi ||Z||^2=-\beta ||Z||^2,$$
which implies
$$-2~\xi \ln f ||Z||^2+g(\tilde \nabla_\xi Z, Z)=-\beta ||Z||^2.$$
Hence, statement $(i)$ follows from the above equation. 

Now, we take the inner product with $Z$ in $(\ref{psb})$ to derive
$$g(\tilde \nabla_\xi \phi Z,Z)+2~ g(\tilde \nabla_\xi  Z,\phi Z)=-\alpha ||Z||^2.$$
This can be written as
$$g(\tilde \nabla_\xi PZ,Z)+g(\tilde \nabla_\xi FZ,Z)+2~g(\tilde \nabla_\xi Z,PZ)+2~g(\tilde \nabla_\xi Z,FZ)=-\alpha ||Z||^2.$$
Hence, by the Gauss formula and part $(ii)$ of Proposition \ref{1}, we get
$$g(\tilde \nabla_\xi FZ,Z)+2~g(\tilde \nabla_\xi Z,FZ)=-\alpha ||Z||^2.$$
Consequently,
$$g(\tilde \nabla_\xi Z,FZ)=-\alpha ||Z||^2.$$
Statement $(ii)$ follows from the above equation by virtue of Gauss formula. This completes the proof.
\end{proof}

In the spirit of the preceding theorem, it is easy, but important, to distinguish other particular case structures. For this, we present the following table:

\begin{center}
\begin{table}[ht]
\centering 
\begin{tabular}{|c c c c c c| } 
 \hline
 $\tilde M^{2l+1}$&$\vline$ & $\xi \ln f=$ &\vline &$g(h(\xi, Z), FZ)=$&\\ [0.9ex] 
\hline\hline 
  
$$ Nearly trans-Sasakian& \vline  &$\beta$&\vline&$-\alpha ||Z||^2$& \\

\hline
Nearly $\alpha$-Sasakian & $\vline$ &0 &\vline&$-\alpha ||Z||^2$&\\

\hline
Sasakian & $\vline$ &0&\vline&$- ||Z||^2$& \\

\hline
Nearly $\beta$-Kenmotsu& $\vline$ & $\beta$&\vline&0&\\ 
 
\hline
Kenmotsu  &$\vline$  & 1&\vline&0&\\

\hline 
Nearly cosymplectic  &$\vline$ &0&\vline&0&\\

\hline 
Cosymplectic &$\vline$ &0&\vline&0&\\
\hline\hline
\end{tabular}
\caption{$\xi \ln f$ and $g(h(\xi, Z), FZ)$ for $N_1\times _fN_2$ in $\tilde M^{2l+1}$, such that $\xi$ is tangent to $N_1$ and $Z$ is tangent to $N_2$.}
\end{table}
\end{center}

Now, assume that the warped product submanifold $N_1\times _fN_2$ in Theorem \ref{U1} is mixed totally geodesic. Thus, from statement $(ii)$ of the same theorem, we have
$$\alpha ||Z||^2=0.$$
This implies that, either $N_2$ is null, or $\alpha=0$; i.e., $\tilde M^{2l+1}$ is not $\alpha$-Sasakian. Therefore, we get the following significant nonexistence result, which will be useful in inequalities of mixed totally geodesic submanifolds in almost contact manifolds.
\begin{proposition}\label{mixtgs}
There is no mixed totally geodesic warped product submanifold in nearly $\alpha$-Sasakian manifolds.
\end{proposition}

In another line of thought, one can easily verify the following lemma.
\begin{lemma}\label{conn}
Let $N_1\times _fN_2$ be a warped product submanifold in almost contact manifolds $\tilde M^{2l+1}$ such that $\xi$ is tangent to the first factor. Then, $g((\tilde \nabla_\xi \phi)Z, \phi Z)=0$ for every $Z\in \Gamma (TN_2).$
\end{lemma}

As another important consequence of Theorem \ref{T1}, we have the following proposition:
\begin{proposition}\label{X1}
For any warped product submanifold $M^n=N_T\times _fN$ of nearly trans-Sasakian manifolds with $\xi$ tangent to the first factor, the followings are true
\begin{enumerate}
\item[(1)] $g(h(X,Y), FZ)=0;$
\item[(2)] $g(h(X,X), \zeta)+ g(h(\phi X,\phi X), \zeta)=0;$
\item [(3)] $ g(h(X,Z), FZ)+ \alpha \eta(X) ||Z||^2=-(\phi X\ln f) ||Z||^2,$
\end{enumerate}
where the vector fields $X,~Y$ are tangent to the first factor, $Z$ is tangent to the second and $\zeta$ is tangent to the normal subbundle $\nu$.
\end{proposition}
\begin{proof}
Statement $(1)$ follows from $(i)$ and $(iii)$ of Theorem \ref{T1}, while $(3)$ is a consequence of $(vi)$ and $(ii)$ of the same theorem. For statement $(2)$ we apply the nearly trans-Sasakian structure for the vector fields $X$ and $\xi$ to obtain the following 
$$2\tilde \nabla_X\phi X= \alpha \biggl(2g(X,X)\xi -2 \eta(X)X\biggr)-2\beta \eta(X)\phi X.$$ 
Taking the inner product with $\zeta$ gives
 $$g(h(X, \phi X), \zeta)+g(\phi h(X,X)\zeta)=0.$$
Replacing $X$ by $\phi X$ gives
$$g(h(-X+\eta(X)\xi, \phi X), \zeta)+g(\phi h(\phi X,\phi X)\zeta)=0.$$
By these two equations and the fact $h(X, \xi)=0$, we obtain the result.
\end{proof}

By means of Propositions \ref{mixtgs} and \ref{X1}, one can easily show that a mixed totally geodesic contact $CR$-warped product submanifold is indeed trivial in both Sasakian and cosymplectic manifolds. Whereas such submanifolds do exist in Kenmotsu manifolds, this is due to the fact $\xi \ln f=1$ for all warped product submanifolds of Kenmotsu manifolds when $\xi$ is tangent to the first factor.

In the sequel, we prove necessary and sufficient conditions for a contact $CR$-submanifold to be locally contact $CR$-warped product in nearly trans-Sasakian manifolds. For long time, mathematicians had have interest to find an analogue of the classical de Rham theorem to warped products, a result was proved by S. Hiepko \cite{nfds}. First, let us recall this result: Let ${\mathcal H}$ be a distribution in the tangent bundle of a Riemannian manifold $M^n$ and let ${\mathcal H}^\perp$ be its orthogonal complementary distribution. Assume that the two distributions are both involutive and the integral manifolds of ${\mathcal H}$ (resp. ${\mathcal H}^\perp$) are extrinsic spheres (resp. totally geodesic). Then, $M^n$ is locally isometric to a warped product $N_1\times _fN_2$. Moreover, if $M^n$ is simply connected and complete, there exists a global isometry of $M^n$ with a warped product. Using this fundamental method we present the following characterization theorem which has been recently published in \cite{zolo}.

\begin{theorem}\label{267}
Every contact $CR$-submanifold $M^n$ of a nearly trans-Sasakian manifold $\tilde M^{2l+1}$ with an involutive distribution $\mathfrak{D}_\perp$ is locally a contact CR-warped product, if and only if the shape operator of $M^n$ satisfies
\begin{equation}\label{1065}
A_{\phi W}X=- (\phi X \mu) W-\alpha \eta  (X) W,~~~~X\in \mathfrak{D}_T\oplus \langle \xi \rangle,~~~W\in \mathfrak{D}_\perp,
\end{equation}
for a smooth function $\mu$ on $M^n$, satisfying $V(\mu) =0$ for each $V\in \mathfrak{D}_\perp.$
\end{theorem}

Observe that the above theorem generalizes many related recent results, for example contact $CR$-warped product of cosymplectic, Sasakian and Kenmotsu manifolds can be characterized in a similar way as above (see, for example \cite{wenal}).

\section{Research problems based on The Results of Previous Sections}

Due to the results of this paper, we hypothesize a pair of open problems.

Firstly, 
\begin{problem}\label{ama1}
Construct discrete examples of contact $CR$-warped product submanifolds of Sasakian manifolds. 
\end{problem}

Secondly, we ask:
\begin{problem}\label{pqm2}
Construct a solid examples of nearly trans-Sasakian manifolds and contact $CR$-warped product submanifolds of such manifolds.
\end{problem}

\vskip.15in
\begin{acknowledgements}
The first author (Abdulqader Mustafa) would like to thank the Palestine Technical University Kadoori, PTUK, for its supports to accomplish this work.
\end{acknowledgements}

\bigskip

{\bfseries{Data availability Statement:}} Not applicable.
 
\bigskip
{\bfseries{Statements and Declarations:}} The authors have no relevant financial or non-financial interests to disclose. The authors did not receive support from any organization for the submitted work.

\end{document}